\documentclass[9pt]{amsart}

\usepackage{amsmath,amsfonts,amsthm}
\usepackage{color}
\usepackage[]{graphics}
\usepackage{verbatim}
\usepackage{eepic,epic}
\usepackage{epsfig,subfigure,epstopdf}
\usepackage[colorlinks,linktocpage,linkcolor=blue]{hyperref}

\newtheorem{theorem}{Theorem}[section]
\newtheorem{lemma}{Lemma}[section]

\newtheorem{example}{Example}[section]

\newtheorem{remark}{Remark}[section]

\numberwithin{equation}{section}

\begin{document}

\newcommand{\al}{\alpha}
\newcommand{\fy}{\varphi}
\newcommand{\la}{\lambda}
\newcommand{\ep}{\epsilon}
\newcommand{\rh}{\varrho}
\newcommand{\vth}{\vartheta}
\newcommand{\vtht}{{\widetilde \vartheta}}
\newcommand{\rlh}{{\widetilde \varrho}}

\def\tribar{\vert\thickspace\!\!\vert\thickspace\!\!\vert}
\def\Etilh{{\bar{E}_h}}
\def\PD{P(\partial_t)}

\def\dH#1{\dot H^{#1}(\Omega)}
\def\normh#1#2{\tribar #1 \tribar_{\dot H^{#2}(\Omega)}}
\def\vecal{{\vec{\al}}}
\def\pa{\partial}
\def\Dal{{\partial_t^\al}}
\def\dal{\bar \partial}
\def\Om{\Omega}
\def\II{(\Om)}
\def\hv{\widehat w_h}
\def\G{\Gamma}
\def\bPtau{\bar\partial_\tau}
\def\pl{\mathrm{Li}}
\newcount\icount
\def\DDelta{{T}}
\def\hy{\widehat y}

\def\DD#1#2{\icount=#1
  \ifnum\icount<1
  \,_{ 0}\kern -.1em D^{#2}_{\kern -.1em x}
  \else
  \,_{x}\kern -.2em D^{#2}_1
  \fi
}

\def\DDRI#1#2{\icount=#1
  \ifnum\icount<1
  \,_{-\infty}^{\kern 1em R}\kern -.2em D^{#2}_{\kern -.1em x}
  \else
  \,_{x}^R \kern -.2em D^{#2}_\infty
  \fi
}

\def\DDR#1#2{\icount=#1
  \ifnum\icount<1
 _{0}^{ \kern -.1em R} \kern -.2em D^{#2}_{\kern -.1em x}
  \else
 _{x}^{ \kern -.1em R} \kern -.2em D^{#2}_{\kern -.1em 1}
  \fi
}

\def\DDCI#1#2{\icount=#1
  \ifnum\icount<1
  \,_{-\infty}^{\kern 1em C}  \kern -.2em D^{#2}_{\kern -.1em x}
  \else
  \,_{x}^C \kern -.2em  D^{#2}_\infty
  \fi
}

\def\DDC#1#2{\icount=#1
  \ifnum\icount<1
  \,_{0}^C \kern -.2em  D^{#2}_{\kern -.1em x}
  \else
  \,_{x}^C \kern -.2em D^{#2}_1
  \fi
}

\def\Hd#1{\widetilde H^{#1}(\Omega)}

\def\Hdi#1#2{\icount=#1
  \ifnum\icount<1
  \widetilde H_{L}^{#2}\II
  \else
  \widetilde H_{R}^{#2}\II
  \fi
}

\def\Cd{\widetilde C_}

\title[L1 Scheme for Time-Fractional Diffusion]
{An analysis of the L1 Scheme for the subdiffusion equation with nonsmooth data}
\author {Bangti Jin \and Raytcho Lazarov \and Zhi Zhou}
\address{Department of Computer Science, University College London, Gower Street,
London WC1E 6BT, UK ({bangti.jin@gmail.com})}
\address{Department of Mathematics, Texas A\&M University, College Station,
TX 77843-3368, USA \\
(lazarov@math.tamu.edu, zzhou@math.tamu.edu)}
\date{started July, 2014; today is \today}

\begin{abstract}
{The subdiffusion equation with a Caputo fractional derivative of order $\alpha\in(0,1)$ in time arises in
a wide variety of practical applications, and it is often adopted to model anomalous subdiffusion processes in
heterogeneous media. The L1 scheme is one of the most popular and successful numerical methods for discretizing
the Caputo fractional derivative in time. The scheme was analyzed earlier independently by Lin and Xu (2007)
and Sun and Wu (2006), and an $O(\tau^{2-\alpha})$ convergence rate was established, under the assumption
that the solution is twice continuously differentiable in time. However, in view of the smoothing property
of the subdiffusion equation, this regularity condition is restrictive, since it does not hold even for the
homogeneous problem with a smooth initial data. In this work, we revisit the error analysis of the scheme,
and establish an $O(\tau)$ convergence rate for both smooth and nonsmooth initial data. The analysis is valid
for more general sectorial operators. In particular, the L1 scheme is applied to one-dimensional space-time
fractional diffusion equations, which involves also a Riemann-Liouville derivative of order $\beta\in(3/2,2)$
in space, and error estimates are provided for the fully discrete scheme. Numerical experiments are provided
to verify the sharpness of the error estimates, and robustness of the scheme with respect to data regularity.
}
\end{abstract}

\maketitle

\section{Introduction}\label{sec:intro}
We consider the model initial--boundary value problem for the following fractional order parabolic
differential equation for $u(x,t)$:
\begin{alignat}{3}\label{eqn:fde}
   \Dal u-\Delta u&= f,&&\quad \text{in  } \Omega&&\quad T \ge t > 0,\notag\\
   u&=0,&&\quad\text{on}\  \partial\Omega&&\quad T \ge t > 0,\\
    u(0)&=v,&&\quad\text{in  }\Omega,&&\notag
\end{alignat}
where $\Omega$ is a bounded convex polygonal domain in $\mathbb R^d\,(d=1,2,3)$ with a boundary
$\partial\Omega$, $v$ is a given function on $\Om$, and $T>0$ is a fixed value.
Here  $\Dal u$ ($0<\al<1$) denotes the left-sided Caputo fractional derivative of order
$\al$ with respect to $t$ and it is defined by
(see, e.g. \cite[pp.\,91]{KilbasSrivastavaTrujillo:2006})
\begin{equation}\label{McT}
   \Dal u(t)= \frac{1}{\Gamma(1-\al)} \int_0^t(t-s)^{-\al}u'(s)\, ds,
\end{equation}
where $\Gamma(\cdot)$ is Euler's Gamma function defined by $\Gamma(x)=\int_0^\infty t^{x-1}e^{-t}dt$.

The model \eqref{eqn:fde} is known to capture well the dynamics of subdiffusion processes, in
which the mean square variance grows at a rate slower than that in a Gaussian process \cite{BouGeo},
and has found a number of applications. For example, subdiffusion has been successfully
used to describe thermal diffusion in media with fractal geometry \cite{Nigmatulin:1986}, highly
heterogeneous aquifer \cite{AdamsGelhar:1992} and underground environmental problem \cite{HatanoHatano:1998}.
At a microscopic level, the particle motion can be described by a continuous time random walk,
in which the waiting time of the particle motion follows a heavy tailed distribution, as opposed
to a Gaussian process, which is characteristic of the normal diffusion equation. The macroscopic
counterpart is a diffusion equation with a Caputo fractional derivative in time, i.e., \eqref{eqn:fde}.

The derivation and study of accurate numerical methods for the model \eqref{eqn:fde}
with provable (possibly optimal-order) error bounds have attracted
considerable interest in recent years, and a number of efficient numerical schemes, notably the finite
difference method, have been developed. There are two predominant approximations: the L1 type
approximation and Gr\"{u}nwald-Letnikov type approximation. The former is essentially of finite
difference nature, whereas the latter is often based on convolution quadrature. We refer
interested readers to \cite[Section 1 and Table 1]{JinLazarovZhou:2014a} for an updated overview of these numerical
methods and relevant convergence rate results.

In this paper, for reasons to be explained below,
we revisit the L1 scheme, which was independently developed and analyzed in \cite{SunWu:2006} and
\cite{LinXu:2007}. To this end, we divide the interval $[0,
T]$ into a uniform grid with a time step size $\tau = T/N$, $N\in\mathbb{N}$, so that $0=t_0<t_1<\ldots<t_N=T$, and $t_n=n
\tau$, $n=0,\ldots,N$. The L1 approximation of the Caputo fractional derivative $\Dal u(x,t_n)$ is given
by \cite[Section 3]{LinXu:2007}
 \begin{equation}\label{eqn:$L^1$approx}
   \begin{aligned}
     \Dal u(x,t_n) &= \frac{1}{\Gamma(1-\al)}\sum^{n-1}_{j=0}\int^{t_{j+1}}_{t_j}
        \frac{\partial u(x,s)}{\partial s} (t_n-s)^{-\al}\, ds \\
     &\approx \frac{1}{\Gamma(1-\al)}\sum^{n-1}_{j=0} \frac{u(x,t_{j+1})-u(x,t_j)}{\tau}\int_{t_j}^{t_{j+1}}(t_n-s)^{-\al}ds\\
     &=\sum_{j=0}^{n-1}b_j\frac{u(x,t_{n-j})-u(x,t_{n-j-1})}{\tau^\alpha}\\
     &=\tau^{-\al} [b_0u(x,t_n)-b_{n-1}u(x,t_0)+\sum_{j=1}^{n-1}(b_j-b_{j-1})u(x,t_{n-j})] =:L_1^n(u).
   \end{aligned}
 \end{equation}
where the weights $b_j$ are given by
\begin{equation*}
b_j=((j+1)^{1-\alpha}-j^{1-\alpha})/\Gamma(2-\al),\ j=0,1,\ldots,N-1.
\end{equation*}
It was shown in \cite[equation (3.3)]{LinXu:2007} (see also \cite[Lemma 4.1]{SunWu:2006}) that the local truncation
error of the L1 approximation is bounded by $c\tau^{2-\al}$ for some constant $c$ depending only on $u$, provided
that the solution $u$ is twice continuously differentiable.
Further, upon discretizing the spatial derivative(s) using the finite difference method, finite element
method or spectral method, we arrive at a fully discrete scheme. Since its first appearance, the L1 scheme
has been extensively used in practice, and currently it is one of the most popular and successful numerical methods
for solving the time fractional diffusion equation \eqref{eqn:fde}, including the case of nonsmooth data
arising in inverse problems (see, e.g., \cite{JinRundell:2012}).

Note that the fractional diffusion operator has only very limited smoothing property, especially for $t$
close to zero. For example, for the homogeneous equation with an initial data
$v \in L^2(\Omega)$, we have the following stability estimate \cite[Theorem 2.1]{SakamotoYamamoto:2011}
\begin{equation*}
  \|\partial_t^\alpha u \|_{L^2\II}\leq ct^{-\alpha}\|v\|_{L^2\II}.
\end{equation*}
That is, the $\alpha$-th order  Caputo derivative is already unbounded, not to mention the high-order
derivatives. In view of the limited  regularity of the solution of \eqref{eqn:fde}, the high regularity
required in the convergence analysis in these useful works is restrictive.

Hence the $C^2$-regularity assumption generally does not hold for problem \eqref{eqn:fde}, and
the case of nonsmooth data is not covered by the existing error analysis. Our numerical
experiments indicate that the $O(\tau^{2-\alpha})$ convergence rate actually does not hold even for
smooth initial data $v$. This is clearly seen from the numerical results in Table \ref{tab:motiv}, where \texttt{rate}
denotes the empirical convergence rate, and the numbers in the bracket denote the theoretical
rates based on the local truncation error. These computational results show that the scheme is
only first-order accurate even when the initial data $v$ is smooth. This observation necessitates revisiting
the convergence analysis of the L1 scheme, especially for the case of nonsmooth problem data.

The goal of this paper is to fill the gap between the existing convergence theory and the numerical experiments, namely,
establishing of optimal error bounds that are expressed directly in terms of the regularity of the problem data.
For the standard parabolic equation, such bounds are well known (see, e.g. \cite[Chapter 7]{thomee2006galerkin}).
To the best of our knowledge, there is no such error analysis of the L1 scheme in the case of nonsmooth data.
Such theory and numerical experiments are provided in this work.

\begin{table}[htb!]
\caption{The $L^2$-norm of the error $\| U_h^n-u(t_n) \|_{L^2\II}$ for problem \eqref{eqn:fde}
with $f\equiv 0$ and (a) $v=\chi_{(0,1/2)}$ and (b) $v=x(1-x)$, at $t=0.1$, computed by the fully
discrete scheme \eqref{eqn:fully} with $h=2^{-12}$.}\label{tab:motiv}
\begin{center}\setlength{\tabcolsep}{7pt}
\vspace{-.3cm}{
     \begin{tabular}{|c|c|cccccc|c|}
     \hline
      $\alpha$ &  $N$  &$10$ &$20$ &$40$ & $80$ & $160$ &$320$ &rate \\
     \hline
     $0.1$ & (a)  &3.30e-4 &1.62e-4 &8.02e-5 &3.99e-5 &1.99e-5 &9.93e-6 &$\approx$ 1.01 (1.90)\\
             & (b)  &4.97e-4 &2.44e-4 &1.21e-4 &6.00e-5 &2.99e-5 &1.49e-5 &$\approx$ 1.00 (1.90)\\
      \hline
     $0.5$ & (a)  &3.04e-3 &1.44e-3 &6.96e-4 &3.41e-4 &1.68e-4 &8.35e-5 &$\approx$ 1.03 (1.50)\\
             & (b)  &4.61e-3 &2.18e-3 &1.05e-3 &5.16e-4 &2.54e-4 &1.34e-5 &$\approx$ 1.02 (1.50)\\
      \hline
     $0.9$ & (a)  &1.31e-2 &6.46e-3 &3.17e-3 &1.56e-3 &7.68e-4 &3.78e-4 &$\approx$ 1.02 (1.10)\\
             & (b)  &1.94e-2 &9.67e-3 &4.78e-3 &2.36e-3 &1.16e-3 &5.75e-4 &$\approx$ 1.02 (1.10)\\
      \hline
     \end{tabular}}
\end{center}
\end{table}

In Theorem \ref{thm:error_fully}, we present an optimal $O(\tau)$ convergence rate for the fully discrete
scheme based on the L1 scheme \eqref{eqn:$L^1$approx} in time and the Galerkin finite element method in space
for both smooth and nonsmooth data, i.e., $v\in L^2\II$ and $Av\in L^2\II$ ($A=-\Delta$ with a homogeneous Dirichlet boundary condition),
respectively. For example, for $v\in L^2\II$ and $U_h^0=P_hv$, for the fully discrete solution $U_h^n$, there holds
\begin{equation*}
   \| u(t_n)-U_h^n \|_{L^2(\Om)} \le c (\tau t_n^{-1} + h^2t_n^{-\al})  \| v\|_{L^2\II}.
\end{equation*}
Surprisingly, for both $v\in L^2\II$ and $Av\in L^2\II$, the error estimate deteriorates
as time $t$ approaches zero, but for any fixed time $t_n>0$, it can achieve a first-order
convergence. Extensive numerical experiments confirm the optimality of the convergence rates.
Our estimates are derived using the techniques developed by \cite{LubichSloanThomee:1996} for
convolution quadrature and in the interesting recent work of \cite{McleanMustapha:2014} on a piecewise constant discontinuous
Galerkin method. The proof essentially boils down to some delicate estimates of the kernel function,
which involves the polylogarithmic function. Finally, we note that our results are applicable to
more general sectorial operators, including the very interesting case of the space time fractional
differential problem involving a Riemann-Liouville derivative in space \cite{JinLazarovPasciakZhou:2014}.

The rest of the paper is organized as follows. In Section \ref{sec:prelim}, we recall preliminaries on
the fully discrete scheme, and derive the solution representation for the semidiscrete and fully discrete
schemes, which play an important role in the error analysis. The full technical details of the convergence
analysis are presented in Section \ref{sec:err}. In Section \ref{sec:space-frac} we
consider the adaptation of the scheme to a one-dimensional time-space fractional differential equation,
and derive optimal convergence rate. It appears to be the first error estimates expressed directly in
terms of the data regularity for such equation with nonsmooth data.  Numerical results are presented in
Section \ref{sec:numer} to confirm the convergence theory and the robustness of the scheme. Throughout,
the notation $c$, with or without a subscript, denotes a generic constant, which may differ at different
occurrences, but it is always independent of the mesh size $h$ and the time step size $\tau$.

\section{Preliminary}\label{sec:prelim}
In this part, we give the semidiscrete and fully discrete schemes, based on a
standard Galerkin method in space and the L1 approximation in time.
\subsection{Semidiscrete scheme}
Since the solution $u:(0,T]\rightarrow L^2\II$ can be analytically extended to the sector
$\{ z\in\mathbb{C};z\neq0,|\arg z|<\pi/2 \}$ \cite[Theorem 2.1]{SakamotoYamamoto:2011}, when $f\equiv0$,
we may apply the Laplace transform to equation \eqref{eqn:fde} to deduce
\begin{equation}\label{eqn:laptrans}
    z^\al \widehat u(z) + A\widehat u(z)=z^{\al-1} v,
\end{equation}
with the operator $A=-\Delta$ with a homogeneous Dirichlet boundary condition. Hence the solution $u(t)$ can be represented by
\begin{equation}\label{eqn:interep}
    u(t)= \frac{1}{2\pi \mathrm{i}} \int_{\Gamma_{\theta,\delta}} e^{zt}(z^{\al}I+A)^{-1}z^{\al-1} v \,dz,
\end{equation}
where the contour $\Gamma_{\theta,\delta}$ is given by
\begin{equation*}
  \Gamma_{\theta,\delta}=\left\{z\in \mathbb{C}: |z|=\delta, |\arg z|\le \theta\right\}\cup
  \{z\in \mathbb{C}: z=\rho e^{\pm i\theta}, \rho\ge \delta\}.
\end{equation*}
Throughout, we choose the angle $\theta\in (\pi/2,\pi)$. Then
$z^{\al} \in \Sigma_{\theta'}$ with $ \theta'=\al\theta< \pi$ for all $z\in\Sigma_{\theta}:
=\{z\in\mathbb{C}: |\arg z|\leq \theta\}$. Then
there exists a constant $c$ which depends only on $\theta$ and $\al$ such that
\begin{equation}\label{eqn:resol}
  \| (z^{\al}I+A)^{-1} \| \le cz^{-\al},  \quad \forall z \in \Sigma_{\theta}.
\end{equation}

Now we introduce the spatial semidiscrete scheme based on the Galerkin finite element method.
Let $\mathcal{T}_h$ be a shape regular and quasi-uniform triangulation of the
domain $\Omega $ into $d$-simplexes, denoted by $\DDelta$. Then
over the triangulation $\mathcal{T}_h$ we define a continuous piecewise linear finite
element space $X_h$ by
\begin{equation*}
  X_h= \left\{v_h\in H_0^1(\Omega):\ v_h|_\DDelta \mbox{ is a linear function},\ \forall \DDelta \in \mathcal{T}_h\right\}.
\end{equation*}
On the space $X_h$, we define the $L^2(\Omega)$-orthogonal projection $P_h:L^2(\Omega)\to X_h$ and
the Ritz projection $R_h:H^1_0(\Omega)\to X_h$, respectively, by
\begin{equation*}
  \begin{aligned}
    (P_h \fy,\chi) & =(\fy,\chi) \quad\forall \chi\in X_h,\\
    (\nabla R_h \fy,\nabla\chi) & =(\nabla \fy,\nabla\chi) \quad \forall \chi\in X_h,
  \end{aligned}
\end{equation*}
where $(\cdot,\cdot)$ denotes the $L^2\II$-inner product. Then the semidiscrete Galerkin scheme
for problem \eqref{eqn:fde} reads: find $u_h(t)\in X_h$ such that
\begin{equation}\label{eqn:fem}
  (\Dal u_h,\chi) + (\nabla u_h,\nabla \chi) = (f,\chi)\quad\forall \chi\in X_h,
\end{equation}
with $u_h(0)=v_h\in X_h$.
Upon introducing the discrete Laplacian $\Delta_h: X_h\to X_h$ defined by
\begin{equation*}
  -(\Delta_h\fy,\chi)=(\nabla\fy,\nabla\chi)\quad\forall\fy,\,\chi\in X_h,
\end{equation*}
the spatial semidiscrete scheme \eqref{eqn:fem} can be rewritten into
\begin{equation}\label{eqn:fdesemidis}
  \Dal u_h(t) +A_h u_h(t) = f_h(t), \,\, t>0
\end{equation}
with $u_h(0)=v_h\in X_h$, $f_h=P_hf$ and $A_h=-\Delta_h$. Like before, the solution $u_h$ to
\eqref{eqn:fdesemidis} with $f_h\equiv0$ can be represented by
\begin{equation}\label{eqn:semi-interep}
    u_h(t)= \frac{1}{2\pi \mathrm{i}} \int_{\Gamma_{\theta,\delta}} e^{zt}(z^{\al}+A_h)^{-1}z^{\al-1} v_h\,dz.
\end{equation}
Further, for later analysis, we let $w_h=u_h-v_h$. Then $w_h$ satisfies the problem:
\begin{equation*}
    \Dal w_h  + A_hw_h = -A_h v_h,
\end{equation*}
with $w_h(0)=0$. The Laplace transform gives
\begin{equation*}
    z^\al \hv(z)+ A_h\hv(z) = -z^{-1} A_hv_h.
\end{equation*}
Hence, $\hv(z)= K_1(z) v_h$, with
\begin{equation*}
K_1(z)=-z^{-1}(z^\al I + A_h)^{-1}A_h,
\end{equation*}
and the desired representation for $w_h(t)$ follows from the inverse Laplace transform
\begin{equation}\label{eqn:semisol}
  w_h(t)=\frac{1}{2\pi \mathrm{i}}\int_{\Gamma_{\theta,\delta}} e^{zt}K_1(z)v_h dz.
\end{equation}

The semidiscrete solution $u_h$ satisfies the following estimates from \cite{JinLazarovZhou:2013}.
The log factor in the estimates in \cite[Section 3]{JinLazarovZhou:2013}
can be removed using the operator trick in \cite[Section 3]{BazhlekovaJinLazarovZhou:2014}.

\begin{theorem}\label{thm:error-semi}
Let $u$ and $u_h$ be the solutions of problems \eqref{eqn:fde} and \eqref{eqn:fdesemidis} with
$f\equiv0$ and $f_h\equiv0$, respectively. Then the following error estimates hold.
\begin{itemize}
  \item[(a)] If $Av\in L^2\II$ and $v_h=R_h v$, then
  \begin{equation*}
   \| u(t)-u_h(t) \|_{L^2(\Om)} \le ch^2  \| Av \|_{L^2\II}.
  \end{equation*}
  \item[(b)] If $v\in L^2(\Omega)$ and $v_h=P_hv$, then
  \begin{equation*}
   \| u(t_n)-u_h(t) \|_{L^2(\Om)} \le ch^2t^{-\al}  \| v\|_{L^2\II}.
  \end{equation*}
\end{itemize}
\end{theorem}

\subsection{Fully discrete scheme}
Now we describe the fully discrete scheme based on the L1 approximation
\eqref{eqn:$L^1$approx}: find $U_h^n\in X_h$ for $n=1,2,\ldots,N$
\begin{equation}\label{eqn:fully}
    (b_0I+\tau^\al A_h)U_h^n=  b_{n-1} U_h^0+ \sum_{j=1}^{n-1}(b_{j-1}-b_j)U_h^{n-j}+\tau^\al F_h^n,
\end{equation}
with $U_h^0=v_h$ and $F_h^n=P_hf(t_n)$. We focus on the homogeneous case, i.e., $f\equiv 0$, and use
the technique in \cite{LubichSloanThomee:1996} and \cite{McleanMustapha:2014} for convergence analysis.
Throughout, we denote by
\begin{equation*}
\widetilde{\omega}(\xi)=\sum_{{j=0}}^\infty\omega_j\xi^j
\end{equation*}
the generating function of a sequence $\{\omega_j\}_{j=0}^\infty$. To analyze the fully discrete scheme \eqref{eqn:fully},
we first derive a discrete analogue of the solution representation \eqref{eqn:semisol}. The fully discrete
solution $W_h^n:=U_h^n-U_h^0$ satisfies the following time-stepping scheme for $n=1,2,...,N$
\begin{equation*}
    L_1^n(W_h)+ A_hW_h^n=- A_hv_h,
\end{equation*}
with $W_h^0=0$. Next multiplying both sides
of the equation by $\xi^n$ and summing from $1$ to $\infty$ yields
\begin{equation*}
  \sum_{n=1}^\infty L_1^n(W_h)\xi^n + A_h \widetilde W_h(\xi) = - \frac{\xi}{1-\xi} A_hv_h.
\end{equation*}
Now we focus on the term $\sum_{n=1}^\infty L_1^n(W_h)\xi^n$. By the definition of the
difference operator $L_1^n$, we have
\begin{equation*}
\begin{split}
  \sum_{n=1}^\infty L_1^n(W_h)\xi^n & = \tau^{-\al}\sum_{n=1}^\infty \left(b_0W_h^n+\sum_{j=1}^{n-1}(b_j-b_{j-1})W_h^{n-j}\right) \xi^n\\
    & = \tau^{-\al} \sum_{n=1}^\infty \left(\sum_{j=0}^{n-1} b_j W_h^{n-j}\right)\xi^n
    - \tau^{-\al}\sum_{n=1}^\infty \left(\sum_{j=1}^{n-1} b_{j-1} W_h^{n-j}\right)\xi^n\\
     &:=I-II.\\
\end{split}
\end{equation*}
Using the fact $W_h^0=0$ and the convolution rule of generating functions (discrete Laplace transform), the first term
$I$ can be written as
\begin{equation*}
I = \tau^{-\al} \sum_{n=1}^\infty (\sum_{j=0}^{n} b_j W_h^{n-j})\xi^n = \tau^{-\al} \widetilde b(\xi)\widetilde W_h(\xi).
\end{equation*}
Similarly, the second term $II$ can be written as
\begin{equation*}
\begin{split}
II = \tau^{-\al} \sum_{n=1}^\infty (\sum_{j=1}^{n} b_{j-1} W_h^{n-j})\xi^n
= \tau^{-\al} \xi \sum_{n=1}^\infty (\sum_{j=0}^{n-1} b_{j} W_h^{n-1-j})\xi^{n-1}= \tau^{-\al}\xi \widetilde b(\xi)\widetilde W_h(\xi).
\end{split}
\end{equation*}
Hence, we arrive at
\begin{equation*}
  \sum_{n=1}^\infty L_1^n(W_h)\xi^n  = \tau^{-\al}(1-\xi)\widetilde b(\xi)\widetilde W_h(\xi).
\end{equation*}
Next we derive a proper representation for $\widetilde b(\xi)$:
\begin{equation*}
  \begin{aligned}
   \widetilde b(\xi) &= \frac{1}{\Gamma(2-\al)}\sum_{{j=0}}^\infty (( j+1 )^{1-\al}-j^{1-\al}) \xi^j\\
    &= \frac{1-\xi}{\xi\Gamma(2-\al)} \sum_{j=1}^\infty j^{1-\al} \xi^j= \frac{(1-\xi)\pl_{\al-1}(\xi)}{\xi\Gamma(2-\al)},
  \end{aligned}
\end{equation*}
where $\pl_{p}(z)$ denotes the polylogarithm function defined by (see \cite{Lewin:1981})
\begin{equation*}
  \pl_p(z)=\sum_{j=1}^\infty \frac{z^j}{j^p}.
\end{equation*}
The polylogarithm function $\pl_p(z)$ is well defined for $|z|<1$, and it can
be analytically continued to the split complex plane $\mathbb{C}\setminus [1,\infty)$; see \cite{Flajolet:1999}.
With $z=1$, it recovers the Riemann zeta function $\zeta(p)=\pl_p(1)$. Therefore,
the fully discrete solution $\widetilde W_h(\xi)$ can be represented by
\begin{equation*}
  \widetilde W_h(\xi)= -\frac{\xi}{1-\xi} \left( \frac{(1-\xi)^2}{\xi\tau^\al\Gamma(2-\al)}\pl_{\al-1}(\xi)+A_h\right)^{-1}A_h v_h.
\end{equation*}
Simple calculation shows that the function $\widetilde W_h(\xi)$ is analytic at $\xi=0$. Hence
the Cauchy theorem implies that for $\varrho$ small enough, there holds
\begin{equation*}
    W_h^n = -\frac{1}{2\pi\mathrm{i}}\int_{|\xi|=\varrho} \frac{1}{(1-\xi)\xi^n}
   \left(\frac{(1-\xi)^2}{\xi\tau^\al\Gamma(2-\al)}\pl_{\al-1}(\xi)+A_h\right)^{-1}A_h v_h\, d\xi.
\end{equation*}
Upon changing variable $\xi=e^{-z\tau}$, we obtain
\begin{equation*}
    W_h^n = -\frac{1}{2\pi\mathrm{i}}\int_{\Gamma^0} e^{zt_{n-1}}  \frac{\tau}{1-e^{-z\tau}} \left(
   \frac{(1-e^{-z\tau})^2}{e^{-z\tau}\tau^\al\Gamma(2-\al)}\pl_{\al-1}(e^{-z\tau})+A_h\right)^{-1}A_h v_h \, dz,
\end{equation*}
where the contour $\Gamma^0:=\{ z=-\ln(\varrho)/\tau+\mathrm{i} y:|y|\le {\pi}/{\tau} \}$ is oriented
counterclockwise. By deforming the contour $\Gamma^0$ to $ \Gamma_\tau:=\{ z\in \Gamma_{\theta,\delta}:
|\Im(z)|\le {\pi}/{\tau} \}$ and using the periodicity of the exponential function, we obtain the
following alternative representation for $W_h^n$
\begin{equation}\label{eqn:fullysol}
    W_h^n = -\frac{1}{2\pi\mathrm{i}}\int_{\Gamma_\tau} e^{zt_{n-1}}  \frac{\tau}{1-e^{-z\tau}} \left( \frac{(1-e^{-z\tau})^2}{e^{-z\tau}\tau^\al\Gamma(2-\al)}\pl_{\al-1}(e^{-z\tau})+A_h\right)^{-1}A_h v_h \, dz.
\end{equation}
This representation forms the basis of the error analysis.

\section{Error analysis of the fully discrete scheme}\label{sec:err}
In this part, we derive optimal error estimates for the fully discrete scheme \eqref{eqn:fully}. The analysis
is based on the representations of the semidiscrete and fully discrete solutions, i.e., \eqref{eqn:semisol}
and \eqref{eqn:fullysol}. Upon subtracting them, we may write the difference between $W_h^n$ and $w_h(t_n)$ as
$$ w_h(t_n)-W_h^n= I+II,$$
where the terms $I$ and $II$ are given by
\begin{equation*}
  I=\frac{1}{2\pi\mathrm{i}}\int_{\Gamma_{\theta,\delta}\backslash\Gamma_\tau} e^{zt_n} K_1(z) v_hdz
\end{equation*}
and
\begin{equation*}
  II=\frac{1}{2\pi\mathrm{i}}\int_{\Gamma_\tau} e^{zt_n} (K_1(z)-e^{-z\tau} K_2(z))v_hdz.
\end{equation*}
Here the kernel functions $K_1(z)$ and $K_2(z)$ are given by
\begin{equation}\label{eqn:KB1}
 K_1(z):= - z^{-1}(z^\al+A_h)^{-1}A_h =: z^{-1} B_1(z)
\end{equation}
and
\begin{equation}\label{eqn:KB2}
  K_2(z):= - \frac{\tau}{1-e^{-z\tau}}
\left( \frac{1-e^{-z\tau}}{\tau^\al}\psi(z\tau)+A_h\right)^{-1}A_h =: \frac{\tau}{1-e^{-z\tau}}B_2(z),
\end{equation}
respectively, and the auxiliary function $\psi$ is define by
\begin{equation*}
  \psi(z\tau)=\frac{e^{z\tau}-1}{\Gamma(2-\al)}\pl_{\al-1}(e^{-z\tau}).
\end{equation*}

Since the function $|e^{-z\tau}|$ is uniformly bounded on the contour $\Gamma_{\tau}$, we deduce
\begin{equation}\label{split1}
\begin{split}
    \|  K_1(z)-e^{-z\tau} K_2(z) \| &\le |e^{-z\tau}|\|  K_1(z)-K_2(z) \| + |1-e^{-z\tau} |\|K_1(z)\| \\
    & \le c\|  K_1(z)-K_2(z) \| + c|z|\tau \|K_1(z)\|\\
    & \le c\|K_1(z)-K_2(z)\| + c\tau,
\end{split}
\end{equation}
where the last line follows from the inequality, in view of \eqref{eqn:resol}:
\begin{equation*}
  \|K_1(z)\| = |z|^{-1}\|-I+z^\alpha(z^\alpha+A_h)^{-1}\|\leq c|z|^{-1}.
\end{equation*}
Thus it suffices to establish a bound on $\|  K_1(z)-K_2(z) \|$, which will be carried out below.
First we give bounds on the function $\chi(z)=\tau^{-1}(1-e^{-z\tau})$.
\begin{lemma}\label{lem:para_equiv1}
Let $\chi(z)=\tau^{-1}(1-e^{-z\tau})$. Then for all $z\in\Gamma_\tau$, there hold for some $c_1,c_2>0$
\begin{equation*}
 |\chi(z)-z|\le c |z|^2\tau \quad \mbox{and}\quad c_1|z| \le |\chi(z)|\le c_2|z|.
\end{equation*}
\end{lemma}
\begin{proof}
We note that $|z|\tau \le \pi/\sin\theta$ for $z \in \Gamma_\tau$. Then the first assertion follows by
\begin{equation*}
  |\chi(z) -z| \le |z|^2\tau \bigg| \sum_{j=0}^\infty \frac{|z|^j\tau^j}{(j+2)!}  \bigg| \le c |z|^2\tau
\quad \text{for}~~  z\in\Gamma_\tau.
\end{equation*}
Next we consider the second claim. The upper bound on $\chi(z)$ is trivial. Thus it suffices to verify
the lower bound. To this end, we split the contour $\Gamma_\tau$ into three disjoint parts $\G_\tau=
\G_\tau^+ \cup \G_{\tau}^c \cup \G_\tau^-$, with $\G_\tau^+$ and $\Gamma^-_\tau$ being the rays in the
upper and lower half planes, respectively, and $\Gamma_\tau^c$ being the circular arc. Here we set
$\xi=-z\tau$ with $\rho\equiv|\xi|\in(0,\pi/\sin\theta)$. We first
consider the case of $z\in \Gamma_\tau^+$, for which $\xi=\rho e^{-i(\pi-\theta)}$, $\rho\in(1,\pi/
\sin\theta)$. Using the trivial inequality $|\cos(\rho\sin\theta)|\le 1$, we obtain
\begin{equation}\label{eqn:arg1}
\begin{split}
  \left| \frac{1-e^{-z\tau}}{z\tau}  \right| & = \left|  \frac{e^{\xi}-1}{\xi}\right|
  = \frac{|e^{-\rho\cos\theta}\cos(\rho\sin\theta)-1-\mathrm{i}e^{-\rho\cos\theta}\sin(\rho\sin\theta)|}{\rho} \\
& \ge \frac{(e^{-2\rho\cos\theta}+1-2e^{-\rho\cos\theta})^{1/2}}{\rho}=\frac{e^{-\rho\cos\theta}-1}{\rho}\ge -\cos\theta >0
\end{split}
\end{equation}
due to positivity and monotonicity of $(e^{-\rho\cos\theta}-1)/\rho$ as a function of $\rho$ over the
interval $(0,\infty)$. The case of $z\in \G_\tau^-$ follows analogously. Last, we consider $z\in \G_\tau^c$,
the circular arc. In this case, by means of Taylor expansion, we have
\begin{equation*}
 \chi(z) = z \left(1 + \sum_{j=1}^\infty (-1)^{j}\frac{z^{j}\tau^j}{(j+1)!}\right).
\end{equation*}
From this and the fact that $\rho=|z\tau|<1$, it follows directly that $|\chi(z)|\geq c|z|$ for $z\in
\Gamma_\tau^c$. This completes the proof of the lemma.
\end{proof}

Then by the trivial inequality $\| B_1(z)  \|\le c$ and Lemma \ref{lem:para_equiv1}, we deduce that
\begin{equation}\label{split2}
\begin{split}
    \|  K_1(z)- K_2(z) \| & \le | z^{-1}-\chi(z)^{-1}| \| B_1(z)\|   +  |\chi(z)|^{-1}\|B_1(z)-B_2(z)\| \\
    & \le \frac{|z-\chi(z)|}{|z\chi(z)|} + c|z|^{-1}\|B_1(z)-B_2(z)\| \\
    &\le c \tau + c|z|^{-1}\|B_1(z)-B_2(z)\| .
\end{split}
\end{equation}
Thus it suffices to establish a bound on $\|B_1(z)-B_2(z)\|$. This will be done using a series of
lemmas. To this end, first we recall an important singular expansion of the function $\pl_p(e^{-z})$
\cite[Theorem 1]{Flajolet:1999}.
\begin{lemma}\label{lem:asym-polylog}
For $p\neq1,2,\ldots$, the function $\mathrm{Li}_p(e^{-z})$ satisfies the singular expansion
\begin{equation}\label{eqn:polylogasymp-0}
  \mathrm{Li}_p(e^{-z})\sim \Gamma(1-p)z^{p-1}+\sum_{k=0}^\infty (-1)^k\zeta(p-k)\frac{z^k}{k!}\quad \mbox{as } z\to 0,
\end{equation}
where $\zeta$ is the Riemann zeta function.
\end{lemma}

\begin{remark}
The singular expansion in Lemma \ref{lem:asym-polylog} is stated only for $z\to0$. However, we note
that the expansion is valid in the sector $\Sigma_{\theta,\delta}$ \cite[pp. 377, proof of Lemma 2]{Flajolet:1999}.
\end{remark}

We shall need the following results. The first result gives the absolute convergence
of a special series involving the Riemann zeta function $\zeta$.
\begin{lemma}\label{lem:pre}
Let $ |z|\le \pi/\sin\theta $ with $\theta \in (\pi/2,5\pi/6)$, and $p=\alpha-1$. Then the series \eqref{eqn:polylogasymp-0}
converges absolutely.
\end{lemma}
\begin{proof}
Using the following well-known functional equation for the Riemann zeta function (see
e.g., \cite{KnoppRobins:2001} for a short proof): for $z\notin \mathbb{Z}$, there holds
\begin{equation*}
    \zeta(1-z)=\frac{2}{(2\pi)^z}\cos\left(\frac{z\pi}{2}\right)\Gamma(z)\zeta(z),
\end{equation*}
we obtain for $p=\al-1\in(-1,0)$
\begin{equation*}
  \begin{aligned}
    \zeta(p-k) &= \zeta(1-(1-p+k))\\
     &=\frac{2}{(2\pi)^{1-p+k}}\cos\left(\frac{(1-p+k)\pi}2\right)\Gamma(1-p+k)\zeta(1-p+k).
  \end{aligned}
\end{equation*}
By Stirling's formula for the Gamma function $\Gamma(x)$, $x\to \infty$ \cite[pp. 257]{AbramowitzStegun:1964}
\begin{equation*}
   \Gamma(x+1) = x^{x+1} e^{-x} \sqrt{\frac{2\pi}{x}}\left(1+O(x^{-1})\right)
\end{equation*}
and that $\zeta(1-p+k) \to 1 $ as $k\to \infty$, we have
\begin{equation*}
   \lim_{k\to \infty}\sqrt[k]{ \frac{|\zeta(p-k)||z|^k}{k!} } \le \frac{1}{2\sin\theta} \quad \forall  |z|\le \pi/\sin\theta.
\end{equation*}
Since for $\theta \in (\pi/2,5\pi/6)$, $2\sin\theta>1$, the series converges absolutely.
\end{proof}

Next we state an error estimate for the function $\frac{1-e^{-z\tau}}{\tau^\al}\psi(z\tau)$ with respect to
$z^\alpha$.
\begin{lemma}\label{lem:pre1}
Let $\psi(z)=\frac{e^{z}-1}{\Gamma(2-\al)}\pl_{\al-1}(e^{-z})$. Then for the choice $\theta\in(\pi/2,5\pi/6)$, there holds
\begin{equation*}
   |\frac{1-e^{-z\tau}}{\tau^\al}\psi(z\tau)-z^\al| \le c |z|^2\tau^{2-\al} \quad \forall z\in \Gamma_\tau.
\end{equation*}
\end{lemma}
\begin{proof}
Upon noting the fact $ 0 \le  |z\tau| \le \pi/\sin\theta$ and using Taylor expansion and
\eqref{eqn:polylogasymp-0}, we deduce that for $z\in\Gamma_\tau$, there holds
\begin{equation*}
  \begin{aligned}
   e^{z\tau}-1 &=\sum_{j=1}^\infty \frac{(z\tau)^j}{j!},\\
   \pl_{\al-1}(e^{-z\tau}) &= \Gamma(2-\al)(z\tau)^{\al-2} + \sum_{k=0}^\infty (-1)^k\zeta(1-\al-k)\frac{(z\tau)^k}{k!}.
  \end{aligned}
\end{equation*}
Hence, the function $\psi(z)$ can be represented by
\begin{equation*}
\begin{split}
    \psi(z\tau) &=  \sum_{j=1}^\infty \frac{(z\tau)^j}{j!}\left[(z\tau)^{\al-2} + \sum_{k=0}^\infty\frac{(-1)^k\zeta(-\al-k)}{\Gamma(2-\al)}\frac{(z\tau)^k}{k!}\right]\\
    &=  \sum_{j=1}^\infty \frac{(z\tau)^{\al+j-2}}{j!} +\sum_{j=1}^\infty \frac{(z\tau)^j}{j!} \sum_{k=0}^\infty\frac{(-1)^k\zeta(-\al-k)}{\Gamma(2-\al)}\frac{(z\tau)^k}{k!},
\end{split}
\end{equation*}
and
\begin{equation*}
\begin{split}
   \frac{1-e^{-z\tau}}{\tau^\al} \psi(z\tau) &=   \tau^{-\al}\sum_{l=0}^\infty (-1)^l\frac{(z\tau)^{l+1}}{(l+1)!}
   \left[  \sum_{j=1}^\infty \frac{(z\tau)^{\al+j-2}}{j!} +\sum_{j=1}^\infty \frac{(z\tau)^j}{j!} \sum_{k=0}^\infty\frac{(-1)^k\zeta(-\al-k)}{\Gamma(2-\al)}\frac{(z\tau)^k}{k!}  \right]\\
   &= z^\al +  \frac{\zeta(\al-1)}{\Gamma(2-\al)}z^2\tau^{2-\al} + O(z^{2+\al}\tau^{2}).
\end{split}
\end{equation*}
In view of the choice $\theta\in(\pi/2,5\pi/6)$ and Lemma \ref{lem:pre}, the bound is uniform,
since the series converges uniformly for $z\in \Gamma_\tau$. Consequently,
\begin{equation*}
  |\frac{1-e^{-z\tau}}{\tau^\al} \psi(z\tau)-z^\al|
  \le |z|^2\tau^{2-\al} \left(  -\frac{\zeta(\al-1)}{\Gamma(2-\al)}+O((z\tau)^\al) \right)\le c |z|^2\tau^{2-\al},
\end{equation*}
from which the desired assertion follows.
\end{proof}

The next result gives a uniform lower bound on the function $\psi(z)$ on the contour $\Gamma_\tau$.
\begin{lemma}\label{lem:pre2}
Let $\psi(z)=\frac{e^{z}-1}{\Gamma(2-\al)}\pl_{\al-1}(e^{-z})$. Then for any $\theta$ close to $\pi/2$,
there holds for any $\delta<\pi/2\tau$
\begin{equation*}
  |\psi(z\tau)|\ge c>0  \quad \forall z\in \Gamma_\tau.
\end{equation*}
\end{lemma}
\begin{proof}
Since for $z\in\Gamma_\tau$, $|\mathfrak{I}z|\le \pi/\tau$ and $z\notin(-\infty,0]$, by
\cite[Lemma 1]{McleanMustapha:2014}, there holds
\begin{equation*}
 \psi(z) = c_\alpha \int_0^\infty \frac{s^{\al-1}}{1-e^{-z-s}}\frac{1-e^{-s}}{s}\,ds,
\end{equation*}
with the constant $c_\alpha={\sin(\pi(1-\al))}/{\pi}$.
To prove the assertion, we again split the contour $\Gamma_\tau$ into $\Gamma=\Gamma_\tau^+\cup\Gamma_\tau^c\cup\Gamma_\tau^-$, where
$\Gamma_\tau^+$ and $\Gamma_\tau^-$ are the rays in the upper and lower half planes, respectively, and
$\Gamma_\tau^c$ is the circular arc of the contour $\Gamma_\tau$, and discuss the three cases separately.
We first consider the case $z\in\Gamma^+_\tau$ and set $z\tau=\rho e^{\mathfrak{i}\theta}=\rho\cos\theta+
\mathrm{i}\rho\sin\theta$ with $\delta<\rho<\pi/\sin\theta$. Upon letting $r=\rho\cos\theta$ and $\phi=\rho\sin\theta$ then
\begin{equation*}
\begin{split}
    \psi(z\tau)&=c_\alpha\int_0^\infty
    \frac{s^{\al-1}}{1-e^{-(r+\mathrm{i}\phi)-s}}\frac{1-e^{-s}}{s}\,ds\\
    &=c_\alpha\int_0^\infty \frac{s^{\al-2}(1-e^{-s})}
    {1-e^{-r-s}\cos\phi+\mathrm{i}e^{-r-s}\sin\phi}\,ds\\
    &=c_\alpha\int_0^\infty
    \frac{s^{\al-2}(1-e^{-s})(1-e^{-r-s}\cos\phi-\mathrm{i}e^{-r-s}\sin\phi)}
    {(1-e^{-r-s}\cos\phi)^2+e^{-2r-2s}\sin^2\phi}\,ds.
\end{split}
\end{equation*}
It suffices to show that the real part
\begin{equation*}
    \Re\psi(z\tau) = c_\alpha\int_0^\infty
    \frac{s^{\al-2}(1-e^{-s})(1-e^{-r-s}\cos\phi)}
    {1-2e^{-r-s}\cos\phi+e^{-2r-2s}}\,ds
\end{equation*}
is bounded from below by some positive constant $c$. First we consider the case  $\phi=\rho\sin\theta\in [\pi/2,\pi]$,
for which $\cos\phi\le0$ and thus
\begin{equation*}
  \begin{aligned}
    0 &< 1 - e^{-r-s}\cos\phi \leq 1-2e^{-r-s}\cos\phi \\
     & \leq 1-2e^{-r-s}\cos\phi+e^{-2r-2s}.
  \end{aligned}
\end{equation*}
Consequently,
\begin{equation*}
\begin{split}
    \Re\psi(z\tau) \ge c_\alpha \int_0^\infty {s^{\al-2}(1-e^{-s})}\,ds=c_0.
\end{split}
\end{equation*}
Next we consider the case $\phi\in(0,\pi/2)$, for which $\cos\phi>0$. Further we fix
$\theta=\pi/2$, and thus $r=\rho\cos\theta=0$ and $e^{-r}\cos\phi=\cos(\rho\sin\theta)
=\cos\rho>0$. Then
\begin{equation*}
   1 - e^{-r-s}\cos\phi > 1 - e^{-s}\quad \mbox{and}\quad
   0\leq 1 - 2e^{-r-s}\cos\phi + e^{-2r-2s} \leq 2,
\end{equation*}
and accordingly the real part $\Re\psi(z\tau)$ simplifies to
\begin{equation*}
\begin{split}
    \Re\psi(z\tau)  &= c_\alpha\int_0^\infty
    \frac{s^{\al-2}(1-e^{-s})(1-e^{-s}\cos\rho)}
    {1-2e^{-s}\cos\rho+e^{-2s}}\,ds\\
    &\ge \frac{c_\alpha}{2}\int_0^\infty s^{\al-2}(1-e^{-s})^2ds\ge c_1.
\end{split}
\end{equation*}
Then by continuity of $\Re\psi(z\tau)$, we may choose an angle $\theta\in(\pi/2,5\pi/6)$ such that
for any $z\in \Gamma^+_\tau$, there holds $ \Re\psi(z\tau) \ge c_2. $ Repeating the above
argument shows also the assertion for the case $z\in \Gamma^-_\tau$. It remains to show the
case $z\in \Gamma_\tau^c$. For any fixed $\rho\in(0,\pi/2)$ and $\theta \in [-\pi/2, \pi/2]$,
$\cos \phi =\cos (\rho\sin\theta)\ge 0 $, $ r = \rho \cos \theta \ge0$. Consequently
\begin{equation*}
  1-e^{-r-s}\cos \phi \geq 1 - e^{-s}\cos \phi \geq 1 - e^{-s},
\end{equation*}
and
\begin{equation*}
  1-2e^{-r-s}\cos\phi+e^{-2r-2s} \leq 1 + e^{-2r-2s} \leq 2.
\end{equation*}
These two inequalities directly imply
\begin{equation*}
\begin{split}
    \Re\psi(z\tau) &= c_\alpha\int_0^\infty
    \frac{s^{\al-2}(1-e^{-s})(1-e^{-r-s}\cos\phi)}
    {1-2e^{-r-s}\cos\phi+e^{-2r-2s}}\,ds\\
    &\ge \frac{c_\alpha}{2} \int_0^\infty
    {s^{\al-2}(1-e^{-s})^2} \,ds \ge c_3.
\end{split}
\end{equation*}
Then by continuity, we may choose an angle $\theta>\pi/2$ such that for $z\in \Gamma^c_\tau$, there holds
$\mathfrak{R}\psi(z\tau) \ge c_4>0.$
\end{proof}

The next result shows a ``sector-preserving'' property of the mapping $\chi_1(z)$: there
exists some $\theta_0<\pi$, such that $\chi_1(z)\in \Sigma_{\theta_0}$ for all $z\in \Sigma_\theta$.
This property plays a fundamental role in the error analysis below.
\begin{lemma}\label{lem:pre3}
Let $\psi(z)=\frac{e^{z}-1}{\Gamma(2-\al)}\pl_{\al-1}(e^{-z})$ and $\chi_1(z)=\frac{1-e^{-z\tau}
}{\tau^\al}\psi(z\tau)$. Then there exists some $\theta_0\in(\pi/2,\pi)$ such that $
\chi_1(z)\in\Sigma_{\theta_0}$ for all $z\in \Sigma_\theta.$
\end{lemma}
\begin{proof}
Like before, for $z\tau=\rho e^{\mathrm{i}\theta}$, we denote by $r=\rho\cos\theta$, $\phi=\rho\sin\theta$ and
$c_\alpha=\sin(\pi(1-\al))/\pi$. Then the real part $\Re\chi_1(z)$ and the imaginary part $\Im\chi_1(z)$ of the kernel
$\chi_1(z)$  are given by
\begin{equation}\label{real}
\begin{split}
    \Re\chi_1(z) = \frac{c_\alpha}{\tau^\al}\int_0^\infty
    \frac{s^{\al-2}(1-e^{-s})(1+e^{-2r-s}-e^{-r-s}\cos\phi-e^{-r}\cos\phi)}
    {1-2e^{-r-s}\cos\phi+e^{-2r-2s}}\,ds
\end{split}
\end{equation}
and
\begin{equation}\label{imagine}
\begin{split}
    \Im\chi_1(z) = \frac{c_\alpha}{\tau^\al}\int_0^\infty
    \frac{s^{\al-2}(1-e^{-s})^2 e^{-r}\sin\phi}
    {1-2e^{-r-s}\cos\phi+e^{-2r-2s}}\,ds,
\end{split}
\end{equation}
respectively. Obviously, for $\theta\in [-\pi/2,\pi/2]$ then $r=\rho\cos\theta\ge0$,
$0<e^{-r}\le 1$ and thus
\begin{equation*}
  \begin{aligned}
    1+e^{-2r-s}&-e^{-r-s}\cos\phi-e^{-r}\cos\phi  \geq
    1 + e^{-2r-s} - e^{-r-s} - e^{-r} \\
  = & (1-e^{-r-s})(1-e^{-r}) \geq (1-e^{-s})(1-e^{-r}).
  \end{aligned}
\end{equation*}
Meanwhile, $r\geq0$ implies
$0 \leq 1-2e^{-r-s}\cos\phi+e^{-2r-2s} \leq  4$, and consequently
\begin{equation*}
\begin{split}
    \Re\chi_1(z)& \ge \frac{c_\alpha(1-e^{-r})}{4\tau^\al}\int_0^\infty
    s^{\al-2}(1-e^{-s})^2\,ds = \frac{c_\alpha'}{\tau^\alpha}(1-e^{-r}),
\end{split}
\end{equation*}
with the constant $c_\alpha' = \frac{c_\alpha}{4}\int_0^\infty s^{\alpha-2}(1-e^{-s})^2ds$.

Next we consider the case $|\theta|>\pi/2$. It suffices to consider the case $\theta>\pi/2$,
and the other case $\theta<-\pi/2$ can be treated analogously. Let $z\tau=
\rho e^{\mathrm{i}\theta}$ with $\rho\in (0,\infty)$. First, clearly, for $\phi=\rho\sin\theta
\in[\pi/2,\pi]$, $\cos\phi\le0$, and there holds
\begin{equation*}
  0 < 1-2e^{-r-s}\cos\phi+e^{-2r-2s} \leq (1+e^{-r-s})^2,
\end{equation*}
and thus
\begin{equation*}
\begin{split}
    \Re\chi_1(z) \ge \frac{c_\alpha}{\tau^\al(1+e^{-r})^2} \int_0^\infty
    {s^{\al-2}(1-e^{-s})}\,ds>0.
\end{split}
\end{equation*}
Second, we consider the case $\phi=\rho\sin\theta \in(0,\pi/2)$. There are two possible
situations: (a) $1-e^{-r}\cos\phi\ge0$ and (b) $1-e^{-r}\cos\phi<0$. In case (a), we have
\begin{equation*}
  \begin{aligned}
   & 1+e^{-2r-s}-e^{-r-s}\cos\phi-e^{-r}\cos\phi  \\
   \geq& 1 + e^{-2r-s}\cos^2\phi - e^{-r-s}\cos\phi - e^{-r}\cos\phi \\
  = & (1-e^{-r-s}\cos\phi)(1-e^{-r}\cos\phi) \geq (1-e^{-s})(1-e^{-r}\cos\phi).
  \end{aligned}
\end{equation*}
Consequently,
\begin{equation*}
\begin{split}
    \Re\chi_1(z) \ge \frac{c_\alpha(1-e^{-r}\cos\phi)}{\tau^\al}\int_0^\infty
    \frac{s^{\al-2}(1-e^{-s})^2}{1-2e^{-r-s}\cos\phi+e^{-2r-2s}}\,ds \ge 0.
\end{split}
\end{equation*}
In case (b), we may further assume $\Re\chi_1(z)<0$, otherwise the statement follows directly. Then appealing to
\eqref{real} and using the trivial inequality $|\cos\phi|\leq1$, we deduce that $e^{-r}>1$ and
\begin{equation*}
1-e^{-r}\cos\phi <0 \quad\text{and}\quad
 e^{-2r}-e^{-r}\cos\phi>0.
\end{equation*}
With the help of these two inequalities, and the assumption $\Re\chi_1(z)<0$, we
arrive at
\begin{equation*}
  \begin{aligned}
    0 & > 1+e^{-2r-s}-e^{-r-s}\cos\phi-e^{-r}\cos\phi \\
      & = (1-e^{-r}\cos\phi) + e^{-s}(e^{-2r}-e^{-r}\cos\phi)\\
      & \geq (\cos\phi - e^{-r}) + e^{-s}(e^{-2r}-e^{-r}\cos\phi)\\
      & \geq e^{-r}(\cos\phi - e^{-r}) + e^{-s}(e^{-2r}-e^{-r}\cos\phi)\\
      & = (1-e^{-s})(e^{-r}\cos\phi-e^{-2r}),
  \end{aligned}
\end{equation*}
where the first and fourth inequalities follow from $\Re\chi_1(z)<0$ and $e^{-r}>1$, respectively.
Consequently,
\begin{equation*}
    |\mathfrak{R}\chi_1(z)| \leq \frac{c_1}{\tau^\al}(e^{-2r}-e^{-r}\cos\phi),
\end{equation*}
with the constant
$$c_1=c_1(r,\phi) = c_\alpha \int_0^\infty \frac{s^{\al-2}(1-e^{-s})^2}
    {1-2e^{-r-s}\cos\phi+e^{-2r-2s}}\,ds.$$
Meanwhile, it follows directly from \eqref{imagine} that
\begin{equation*}
    |\mathfrak{I}\chi_1(z)| = \frac{c_1}{\tau^\al}e^{-r}\sin\phi.
\end{equation*}
Therefore,
\begin{equation*}
   \frac{|\Im\chi_1(z)|}{|\Re\chi_1(z)|} \ge \frac{\sin(\rho\sin\theta) }{e^{-\rho\cos\theta}-\cos(\rho\sin\theta)}=:g(\rho).
\end{equation*}
Now set $g_1(\rho)=\sin(\rho\sin\theta)$ and $g_2(\rho)=e^{-\rho\cos\theta}-\cos(\rho\sin\theta)$.
Since for $\rho\in(0,\pi/(2\sin\theta))$ and $\theta>\pi/2$,
\begin{equation*}
 \lim_{\rho\to0}g(\rho)=- \tan\theta, \quad g_1(\rho),g_2(\rho)\ge0,\quad g_1'(\rho)\le0\quad \text{and}\quad g_2'(\rho)\ge0,
\end{equation*}
i.e., the function $g(\rho)$ is monotonically decreasing on the interval $[0,\pi/2\sin\theta]$, we deduce
\begin{equation*}
\inf_{\rho\in(0,\pi/(2\sin\theta))}g(\rho) = g(\pi/(2\sin\theta))=e^{\pi\cot\theta/2}>0.
\end{equation*}
This completes the proof of the lemma.
\end{proof}

\begin{remark}\label{rem:sector}
By the proof of Lemma \ref{lem:pre3}, the sector $\Sigma_{\theta_0}$ depends on the choice of the angle $\theta$.
For $\theta\to \pi/2$, it is contained in the sector $ \Sigma_{3\pi/4-\epsilon}$, for any $\epsilon>0$.
\end{remark}

\begin{lemma}\label{lem:errK}
Let $\theta$ be close to $\pi/2$, and $\delta<\pi/2\tau$. Then for $K_1(z)$ and $K_2(z)$
defined in \eqref{eqn:KB1} and \eqref{eqn:KB2}, respectively, there holds
\begin{equation*}
    \| K_1(z)-K_2(z) \| \le c \tau \quad \forall z\in \Gamma_\tau.
\end{equation*}
\end{lemma}
\begin{proof}
For the operators $B_1(z)$ and $B_2(z)$ defined in \eqref{eqn:KB1} and \eqref{eqn:KB2}, respectively,
we have
\begin{equation*}
  B_1(z)=-I+ z^\al(z^\al+A_h)^{-1}\quad\mbox{and}\quad
  B_2(z)=-I + \chi_1(z)(\chi_1(z)+A_h)^{-1},
\end{equation*}
where $\chi_1(z)=\frac{1-e^{-z\tau}}{\tau^\al}\psi(z\tau)$.
Then by Lemma \ref{lem:pre1}
\begin{equation*}
\begin{split}
 \|  B_1(z)-B_2(z)\|&\leq |\chi_1(z)-z^\al| \|(\chi_1(z)+A_h)^{-1}\| + |z|^\al \|(\chi_1(z)+A_h)^{-1}-(z^\al+A_h)^{-1}\|\\
 & \le  c|z|^{2} \tau^{2-\al} \|(\chi_1(z)+A_h)^{-1}\|+|z|^\al |\chi_1(z)-z^\al|\|(\chi_1(z)+A_h)^{-1}(z^\al+A_h)^{-1}\|\\
 &\le c |z|^2\tau^{2-\al} \|(\chi_1(z)+A_h)^{-1}\|.
\end{split}
\end{equation*}
Now we note that
\begin{equation*}
 |\chi_1(z)|=|\chi(z)|\tau^{1-\al} |\psi(z\tau)| \ge c|z|\tau^{1-\al}.
\end{equation*}
By Lemma \ref{lem:pre3}, $\chi(z) \in \Sigma_{\theta_0}$ for some $\theta_0\in (\pi/2,\pi)$.
Thus by Lemma \ref{lem:pre2} and the resolvent estimate \eqref{eqn:resol}, we have
\begin{equation*}
\begin{split}
 \|  K_1(z)-K_2(z)\|&\le c|z|^{-1}|z|^2\tau^{2-\al} |\chi_1(z)|^{-1} + c\tau \le c\tau,
\end{split}
\end{equation*}
and the desired estimate follows immediately.
\end{proof}

Now we can state an error estimate for the discretization error in time for nonsmooth initial data,
i.e., $v\in L^2(\Omega)$.
\begin{theorem}\label{thm:error-nonsmooth}
Let $u_h$ and $U_h^n$ be the solutions of problems \eqref{eqn:fdesemidis} and \eqref{eqn:fully} with
$v\in L^2(\Omega)$, $U_h^0= v_h = P_hv$ and $f\equiv0$, respectively. Then there holds
\begin{equation*}
   \| u_h(t_n)-U_h^n \|_{L^2(\Om)} \le c \tau t_n^{-1}  \| v\|_{L^2\II}.
\end{equation*}
\end{theorem}
\begin{proof}
It suffices to bound the terms $I$ and $II$.
With the choice $\delta=t_n^{-1}$ and \eqref{eqn:resol}, we arrive at the following bound for the term $I$
\begin{equation}\label{eqn:errI}
\begin{split}
    \| I \|_{L^2\II}
    &\le c \tau \| v_h\|_{L^2\II} \left(\int_{1/t_n}^{\pi/(\tau\sin\theta)}  e^{r t_n\cos\theta}   \,dr +
    \int_{-\theta}^{\theta} e^{\cos\psi} t_n^{-1}\,d\psi\right) \\
    & \le c t_n^{-1}\tau \|v_h\|_{L^2\II}.
\end{split}
\end{equation}
By Lemma \ref{lem:errK} and direct calculation, we bound the term $II$ by
\begin{equation}\label{eqn:errII}
\begin{split}
    \| II \|_{L^2\II} &\le c \int_{\pi/(\tau\sin\theta)}^\infty  e^{rt_n\cos\theta } r^{-1} \,dr \| v_h\|_{L^2\II}\\
    & \le c\tau \|v_h\|_{L^2\II} \int_{0}^\infty  e^{rt_n\cos\theta} \,dr \le c\tau t_n^{-1} \| v_h\|_{L^2\II}.
\end{split}
\end{equation}
Combining estimates \eqref{eqn:errI} and \eqref{eqn:errII} yields
$ 
   \| w_h(t_n)-W_h^n \|_{L^2(\Om)} \le c \tau t_n^{-1}  \|v_h\|_{L^2\II}
$ 
and the desired result follows from the identity $U_h^n-u_h(t_n)=W_h^n-w_h(t_n)$
and the stability of the projection $P_h$ in $L^2(\Omega)$.
\end{proof}

\begin{remark}\label{rem:stab}
The $L^2(\Omega)$ stability of the L1 scheme follows directly from Theorem \ref{thm:error-nonsmooth}.
\end{remark}

Next we turn to smooth initial data, i.e., $v\in D(A)=H^2\II\cap H_0^1\II$. To this end, we first
state an alternative estimate on the solution kernels.

\begin{lemma}\label{lem:errK2}
Let $\theta$ be close to $\pi/2$, and $\delta<\pi/2\tau$. Further,
let  $K_1^s(z)=-z^{-1}(z^\al+A_h)^{-1}$ and $K_2^s(z)=-\chi(z)^{-1}(\chi_1(z)+A_h)^{-1}$.
Then for any $z\in \Sigma_{\delta,\theta}$, there hold
\begin{equation*}
\| K_1^s(z)-K_2^s(z) \|  \le  c|z|^{-\al}\tau.
\end{equation*}
\end{lemma}
\begin{proof}
Like before, we set $B_1^s(z)=-(z^\al+A_h)^{-1}$ and $B^s_2(z)=-(\chi_1(z)+A_h)^{-1}$.
Then by Lemmas \ref{lem:pre1} and \ref{lem:pre2}, we have
\begin{equation*}
\begin{split}
\| B_1^s(z)-B_2^s(z) \|  &\le  |\chi_1(z)-z^\al| \|  (z^\al+A_h)^{-1} \|   \|  (\chi_1(z)+A_h)^{-1} \| \\
  &\le c |z|^2\tau^{2-\al}|z|^{-\al} |\chi_1(z)|^{-1}\le c|z|^{1-\al}\tau.
\end{split}
\end{equation*}
The rest follows analogously to the derivation \eqref{split1}.
\end{proof}

Now we can state an error estimate for $Av \in L^2\II$.
\begin{theorem}\label{thm:error-smooth}
Let $u_h$ and $U_h^n$ be the solutions of problems \eqref{eqn:fdesemidis} and \eqref{eqn:fully} with
$v\in \dH 2$, $U_h^0= v_h=R_hv$ and $f\equiv0$, respectively. Then there holds
\begin{equation*}
   \| u_h(t_n)-U_h^n \|_{L^2(\Om)} \le c \tau t_n^{\al-1} \| Av \|_{L^2\II}.
\end{equation*}
\end{theorem}
\begin{proof}
Let  $K_1^s(z)=-z^{-1}(z^\al+A_h)^{-1}$ and $K_2^s(z)=-\chi(z)^{-1}(\chi_1(z)+A_h)^{-1}$.
Then we can rewrite the error as
\begin{equation}\label{eqn:Vsplit2}
\begin{split}
   w_h(t_n) - W_h^n=&\frac{1}{2\pi\mathrm{i}}\int_{\Gamma_{\theta,\delta}\backslash\Gamma_\tau} e^{zt_n} K_1^s(z)A_hv_hdz\\
&+\frac{1}{2\pi\mathrm{i}}\int_{\Gamma_\tau} e^{zt_n} (K^s_1(z)-e^{-z\tau} K^s_2(z))A_hv_hdz = I+II.
\end{split}
\end{equation}
By Lemma \ref{lem:errK2} we have for $z\in\G_\tau$
\begin{equation*}
  \| K^s_1(z)-e^{-z\tau} K^s_2(z) \| \le c|z|^{-\al}\tau.
\end{equation*}
By setting $\delta=1/t_n$ and for all $z\in \G_{\delta,\theta}$,
we derive the following bound for the term $II$
\begin{equation}\label{eqn:errIsmooth}
\begin{split}
    \| II \|_{L^2\II} &\le c \tau \| A_hv_h \|_{L^2\II}\left(\int_{1/t_n}^{\pi/(\tau\sin\theta)}  e^{r t_n\cos\theta} r^{-\al}dr
     + \int_{-\theta}^{\theta}  e^{\cos\psi} t_n^{\al-1}d\psi \right)\\
     &\le ct_n^{\al-1}\tau \|A_hv_h\|_{L^2\II}.
\end{split}
\end{equation}
Now \eqref{eqn:resol} implies that
for all $z\in \G_{\delta,\theta}$
\begin{equation}\label{eqn:errIIsmooth}
\begin{split}
    \| I \|_{L^2\II} &\le c \| A_hv_h \|_{L^2\II} \int_{\pi/(\tau\sin\theta)}^\infty  e^{r t_n\cos\theta} r^{-\al-1} \,dr\\
    &\le c \tau \| A_hv_h \|_{L^2\II} \int_{0}^\infty  e^{r t_n\cos\theta} r^{-\al} \,dr
     \le  c\tau t_n^{\al-1} \| A_hv_h \|_{L^2\II}.
\end{split}
\end{equation}
Then the desired result follows directly from \eqref{eqn:errIsmooth}, \eqref{eqn:errIIsmooth} and the
identities $U_h^n-u_h(t_n)=W^n-w_h(t_n)$ and $A_hR_h=P_hA$.
\end{proof}

\begin{remark}\label{rem:singularity}
The convergence behavior of the L1 scheme is identical with that for the convolution quadrature generated
by the backward Euler method, which also converges at an $O(\tau)$ rate \cite{JinLazarovZhou:2014a}.
In particular for smooth initial data $v\in D(A)$, the time discretization error by both schemes
contains a singularity $t_n^{\alpha-1}$. This singularity reflects the limited smoothing property
of the solution $u$ \cite[Theorem 2.1]{SakamotoYamamoto:2011}
\begin{equation*}
  \|\Dal u(t)\|_{L^2(\Omega)}\leq c\|Av\|_{L^2(\Omega)},
\end{equation*}
whereas the first order derivative $u'(t)$ is unbounded at $t=0$.
\end{remark}

\begin{example}
To illustrate the convergence rate in Theorem \ref{thm:error-smooth}, we give a trivial example.
Consider the following initial value problem for the fractional ordinary differential equation:
\begin{equation*}
   \Dal u + u = 0, ~~\forall t>0 , \quad \text{with} ~~ u(0)=1.
\end{equation*}
The exact solution $u$ at $t=\tau$ is given by $u(\tau)=E_{\al,1}(-\tau^\al)$, where $E_{\al,1}(z)=
\sum_{k=0}^\infty z^k/\Gamma(\alpha k +1)$ is the Mittag-Leffler function. For small $\tau$, the L1 scheme at the
first step is given by
\begin{equation*}
  U^1=(1+\Gamma(2-\al)\tau^\al)^{-1}= 1+ \sum_{n=1}^\infty (-1)^n (\Gamma(2-\al)\tau^\al)^n.
\end{equation*}
Then the difference between $U^1$ and $u(\tau)$ is given by
\begin{equation*}
   u(\tau)-U^1 = ( \Gamma(2-\al)-\Gamma(\al+1)^{-1} ) \tau^{\al} + c_\tau\tau^{2\al},
\end{equation*}
with $c_\tau = \sum_{n=2}^{\infty} (-1)^n (\Gamma(n\al+1)^{-1}-\Gamma(2-\al)^{n})\tau^{(n-2)\al} $.
Since $|c_\tau| \le c_0$ for small $\tau$, 
we deduce that
\begin{equation*}
  | u(\tau)-U^1 | \sim \tau^{\al} = t_1^{\al-1} \tau.
\end{equation*}
This confirms the convergence order in Theorem \ref{thm:error-smooth}.
\end{example}

Last, the error estimates for the fully discrete scheme \eqref{eqn:fully}
follow from Theorems \ref{thm:error-semi}, \ref{thm:error-nonsmooth} and
\ref{thm:error-smooth} and the triangle inequality.

\begin{theorem}\label{thm:error_fully}
Let $u$ and $U_h^n$ be the solutions of problems \eqref{eqn:fde} and \eqref{eqn:fully} with
$U_h^0= v_h$ and $f\equiv0$, respectively. Then the following estimates hold.
\begin{itemize}
  \item[(a)] If $v\in D(A)$ and $v_h=R_h v$, then for $n \ge 1$
  \begin{equation*}
   \| u(t_n)-U_h^n \|_{L^2(\Om)} \le c(\tau t_n^{\al-1} + h^2)  \| Av \|_{L^2\II}.
  \end{equation*}
  \item[(b)] If $v\in L^2(\Omega)$ and $v_h=P_hv$, then for $n \ge 1$
  \begin{equation*}
   \| u(t_n)-U_h^n \|_{L^2(\Om)} \le c (\tau t_n^{-1} + h^2t_n^{-\al})  \| v\|_{L^2\II}.
  \end{equation*}
\end{itemize}
\end{theorem}

\begin{remark}\label{rem:interp}
For $v\in D(A)$, we can also choose $v_h=P_h v$ by the stability
of the L1 scheme. Hence, by interpolation we deduce
\begin{equation*}
   \| u(t_n)-U_h^n \|_{L^2(\Omega)} \le c (\tau t_n^{-1+\al\sigma} +h^{2} t_n^{-\al(1-\sigma)}) \|A^\sigma v\|_{L^2\II},
     \quad 0 \le \sigma \le 1.
\end{equation*}
\end{remark}

\section{Time-space fractional differential problem}\label{sec:space-frac}
The convergence theory developed in Section \ref{sec:err} may be extended to more general sectorial operators $A$,
i.e., (a) The resolvent set $\rho(A)$ contains the sector $\left\{ z: \theta\leq |\arg z|\leq \pi\right\}$ for
some $\theta\in(0,\pi/4)$; (b) $\| (z I+A)^{-1}  \| \le M/|z|$ for $z \in \Sigma_{\pi-\theta}$ and some
constant $M$. The technical restriction $\theta\in(0,\pi/4)$ stems from Remark \ref{rem:sector}. This in particular
covers the Riemann-Liouville fractional derivative of order $\beta\in(3/2,2)$; see Lemma \ref{lem:riem-sectorial}
below. Specifically, we consider the following one-dimensional space-time fractional differential equation
\begin{alignat}{3}\label{eqn:fde-sp-frac}
   \Dal u- {\DDR0 \beta} u&= f,&&\quad \text{in  } \Omega=(0,1)&&\quad T \ge t > 0,\notag\\
   u&=0,&&\quad\text{on}\  \partial \Omega&&\quad T \ge t > 0,\\
    u(0)&=v,&&\quad\text{in  }\Omega,&&\notag
\end{alignat}
with $\al\in(0,1)$ and $\beta\in(3/2,2)$. Here $\DDR0 \beta$ with $n-1 < \beta < n$, $n\in \mathbb{N}$,
denotes the left-sided Riemann-Liouville fractional derivative $\DDR0\beta u$ of order $\beta$ defined
by \cite[pp. 70]{KilbasSrivastavaTrujillo:2006}:
\begin{equation}\label{Riemann}
  \DDR0\beta u =\frac{1}{\Gamma(n-\beta)}\frac {d^n} {d x^n} \int_0^x (x-s)^{n-\beta-1}u(s)ds.
\end{equation}
The right-sided version of Riemann-Liouville fractional derivative is defined analogously
\begin{equation*}
  \DDR1\beta u =\frac{(-1)^n}{\Gamma(n-\beta)}\frac {d^n} {d x^n} \int_x^1 (s-x)^{n-\beta-1}u(s)ds.
\end{equation*}
The model \eqref{eqn:fde-sp-frac} is often adopted to describe anomalous diffusion process
involving both long range interactions and history mechanism.

The variational formulation of \eqref{eqn:fde-sp-frac} is to find $u \in \Hd {\beta/2} \equiv H^{\beta/2}_0(\Omega)$ such that (see \cite{JinLazarovPasciak:2013a})
\begin{equation}\label{eqn:var-sp-frac}
   (\Dal u ,\fy) + A(u,\fy) = (f,\fy) \quad \forall \fy \in \Hd {\beta/2},
\end{equation}
with $u(0)=v$, where the sesquilinear form $A(\cdot,\cdot)$ is given by
\begin{equation*}
   A(\fy,\psi)=-\left({\DDR0{\beta/2}} \fy,\ \DDR1{\beta/2}\psi\right).
\end{equation*}
It is known (see \cite[Lemma 3.1]{ErvinRoop:2006} and \cite[Lemma 4.2]{JinLazarovPasciak:2013a}) that
the sesquilinear form $A(\cdot,\cdot)$ is coercive and bounded on the space $\Hd{\beta/2}$.
Then Riesz representation theorem implies that there exists a unique bounded linear
operator $\widetilde A: \Hd {\beta/2} \rightarrow H^{-\beta/2}\II$ such that
\begin{equation*}
   A(\fy,\psi)= \langle \widetilde A \fy, \psi \rangle ,\quad \forall \fy,\psi \in \Hd {\beta/2}.
\end{equation*}
Define $D(A)=\{ \psi \in \Hd {\beta/2}: \widetilde A \psi \in L^2\II\}$
and an operator $A : D(A)\rightarrow L^2\II$ by
\begin{equation}\label{eqn:A}
A(\fy,\psi)=(A\fy,\psi),\ \fy \in D(A),\, \psi \in \Hd {\beta/2}.
\end{equation}
We recall that the domain $D(A)$ consists of functions of the form ${\,_0\hspace{-0.3mm}I^\beta_x}
f - ({\,_0\hspace{-0.3mm}I^\beta_x} f)(1)x^{\beta-1}$, where $f\in L^2\II$ \cite{JinLazarovPasciak:2013a}. The
term $x^{\beta-1}\in\Hdi0{\beta-1+\delta}$, $\delta\in[1-\beta/2,1/2)$, appears because it is in the
kernel of the operator $\DDR0\beta$. The presence of the term $x^{\beta-1}$ indicates that the solution $u$
usually can only have limited regularity.

\begin{lemma}\label{lem:riem-sectorial}
For $\beta\in(3/2,2)$,
the resolvent set $\rho(A)$ contains the sector $\{z: \theta\leq |\arg(z)|\leq \pi\}$ for any $\theta\in((2-\beta)\pi/2,\pi/4)$.
\end{lemma}
\begin{proof}
Let $ \widetilde u$ be the zero extension of $u$. Recall that for $s\in(1/2,1)$ and $\Omega=(0,1)$,
\begin{equation*}
  \|u \|_{\Hd{s}} = \| |\xi|^s \mathcal{F}(\widetilde u)(\xi)  \|_{L^2(\mathbb{R})},
\end{equation*}
where $\mathcal{F}(\widetilde u)$ is the Fourier transform of $\widetilde u$, is a consistent and well-defined norm
on $\Hd{s}$. Further, we note that for $u\in \Hd{\beta/2}$ (see \cite{ErvinRoop:2006} and \cite{JinLazarovPasciak:2013a})
\begin{equation*}
  \Re (A(u,u))\ge c_0 \| u \|_{\Hd{\beta/2}}^2 \quad \text{with} ~~c_0= \cos((2-\beta/2)\pi).
\end{equation*}
Further, we recall the fact that for $u\in C^\infty_0(\Omega)$ and $s\in(1/2,1)$
\begin{equation*}
\begin{split}
 \|  \DDR0 s u \|_{L^2\II}&= \|  \DDRI0 s \widetilde u \|_{L^2\II} \leq \|\DDRI 0 s \widetilde u\|_{L^2(\mathbb{R})} \\
   &= \|\mathcal{F}(\DDRI0 s \widetilde u)\|_{L^2(\mathbb{R})} = \|\mathcal{F}(\DDRI0 s \widetilde u)\|_{L^2(\mathbb{R})}\\
  &= \||\xi|^s\mathcal{F} (\widetilde u)(\xi)\|_{L^2(\mathbb{R})} = \|  u \|_{\Hd{s}}.
\end{split}
\end{equation*}
By the density of $C_0^\infty\II$ in $\Hd{s}$ for $s\in(1/2,1)$ we obtain for all $u\in \Hd{\beta/2}$
\begin{equation*}
  \|  \DDR0 {\beta/2} u \|_{L^2\II} \leq \|  u \|_{\Hd{\beta/2}}.
\end{equation*}
Likewise, the right sided case follows:
\begin{equation*}
  \|  \DDR1 {\beta/2} u \|_{L^2\II} \leq \|  u \|_{\Hd{\beta/2}}.
\end{equation*}
Thus for $u,v\in \Hd{\beta/2}$, there holds
\begin{equation*}
  |A(u,v)| \le c_1 \| u \|_{\Hd{\beta/2}}\| v \|_{\Hd{\beta/2}}  \quad \text{with} ~~c_1= 1.
\end{equation*}
Then by Lemma 2.1 of \cite{JinLazarovPasciakZhou:2014}, we conclude that the
resolvent set $\rho(A)$ contains the sector $\left\{ z: \theta \leq |\arg z|\leq \pi
\right\}$ for all $\theta\in((2-\beta/2)\pi,\pi/2)$. In particular, for $\beta>3/2$, we may choose $\theta\in((2-\beta/2)\pi,\pi/4)$.
\end{proof}

By Lemma \ref{lem:riem-sectorial} and Remark \ref{rem:sector}, we can apply the theory in Sections
\ref{sec:prelim} and \ref{sec:err} to derive a fully discrete scheme based on the L1 scheme in time and
the Galerkin finite element approximation in space.
First we partition the unit interval $\Omega$ into a uniform mesh with a mesh size $h=1/M$.
We then define $V_h$ to be the set of continuous functions in $V$ which are
linear when restricted to the subintervals $[x_i,x_{i+1}]$, $i=0,\ldots,M-1$.
Further, we define the discrete operator $A_h: V_h \rightarrow V_h$ by
\begin{equation*}
(A_h \fy, \chi) = A(\fy,\chi)\quad \forall \fy,\chi \in V_h.
\end{equation*}
The corresponding Ritz projection $R_h:\Hd{\beta/2} \rightarrow V_h$ is defined by
\begin{equation}\label{eqn:Ritz}
    A(R_h\psi,\chi)=A(\psi,\chi) \quad \forall \psi\in \Hd{\beta/2}, \ \chi\in V_h.
\end{equation}
Then the fully discrete scheme for problem \eqref{eqn:fde-sp-frac} based on the L1
approximation \eqref{eqn:$L^1$approx} reads: find $U_h^n$ for $n=1,2,\ldots,N$
\begin{equation}\label{eqn:fully-sp-frac}
    (I+\tau^\al A_h)U_h^n=  b_n U_h^0+ \sum_{j=1}^n(b_{j-1}-b_j)U_h^{n-j}+\tau^\al F_h^n,
\end{equation}
with $U_h^0=v_h$ and $F_h^n=P_hf(t_n)$.

Last, we state the error estimates for the fully discrete scheme \eqref{eqn:fully-sp-frac}.
These estimates follow from Theorems \ref{thm:error-nonsmooth} and \ref{thm:error-smooth}
and the error estimates for the semidiscrete Galerkin scheme, which can be proved using the
operator trick in \cite{JinLazarovPasciakZhou:2014}, and thus the proof is omitted.
\begin{theorem}\label{thm:error_fully-sp-frac}
Let $u$ and $U_h^n$ be the solutions of problems \eqref{eqn:fde-sp-frac} and \eqref{eqn:fully-sp-frac} with
$U_h^0= v_h$ and $f\equiv0$, respectively. Then for $\delta\in[1-\beta/2,1/2)$ the following estimates hold.
\begin{itemize}
  \item[(a)] If $v\in D(A)$ and $v_h=R_h v$, then for $n \ge 1$
  \begin{equation*}
   \| u(t_n)-U_h^n \|_{L^2(\Om)} \le c(\tau t_n^{\al-1} + h^{\beta-1+\delta})  \| A v \|_{L^2\II}.
  \end{equation*}
  \item[(b)] If $v\in L^2(\Omega)$ and $v_h=P_hv$, then for $n \ge 1$
  \begin{equation*}
   \| u(t_n)-U_h^n \|_{L^2(\Om)} \le c(\tau t_n^{-1} +  h^{\beta-1+\delta} t_n^{-\al})  \| v \|_{L^2\II}.
  \end{equation*}
\end{itemize}
\end{theorem}

\section{Numerical experiments and discussions}\label{sec:numer}
Now we present numerical results to verify the convergence theory in Sections \ref{sec:err} and \ref{sec:space-frac},
i.e., the $O(\tau)$ convergence rate. We consider the subdiffusion case \eqref{eqn:fde}
and time-space fractional case \eqref{eqn:fde-sp-frac} separately. For each model, we consider
the following two initial data:
\begin{itemize}
  \item[(a)] $\Omega=(0,1)$, and $v=\sin(2\pi x) \in H^2(\Omega)\cap H_0^1(\Omega)$.
  \item[(b)] $\Omega=(0,1)$, and $v=x^{-1/4}\in H^{{1/4}-\epsilon}(\Omega)$, with $\epsilon\in(0,1/4)$;
\end{itemize}
We measure the error $e^n = u(t_n)-U_h^n$ by the normalized errors $\| e^n \|_{L^2\II}/
\| v \|_{L^2\II}$. In the computations, we divide the unit interval $\Omega=(0,1)$ into $M$
equally spaced subintervals with a mesh size $h=1/M$. Likewise, we fix the time step size
$\tau$ at $\tau=t/N$, where $t$ is the time of interest. In this work, we only examine the temporal convergence rate, since
the space convergence rate has been examined earlier in \cite{JinLazarovZhou:2013,JinLazarovPasciakZhou:2014}.
To this end, we take a small mesh size $h=2^{-13}$, so that the spatial discretization error is negligible.

\subsection{Subdiffusion.}
The exact solution can be written explicitly as an infinite series involving
the Mittag-Leffler function $E_{\al,\beta}(z)$ defined by
\begin{equation*}
  E_{\alpha,\beta}(z) = \sum_{k=0}^\infty \frac{z^k}{\Gamma(\alpha k+\beta)};
\end{equation*}
see \cite{JinLazarovZhou:2013} for the details. The Mittag-Leffler function $E_{\alpha,\beta}(z)$ can be evaluated efficiently
by the algorithm developed in \cite{Seybold:2008}. The numerical results for cases (a) and (b) are shown
in Table \ref{tab:time-error}, where \texttt{rate} refers to the empirical convergence rate, and the number
in the bracket refers to the theoretical rate from Theorem \ref{thm:error_fully}. The results fully confirm the theoretical prediction: for
both smooth and nonsmooth data, the fully discrete solution $U_h^n$ converges at a rate $O(\tau)$, and the rate
is independent of the fractional order $\alpha$. Further, for fixed $t$, the error increases with the fractional order
$\alpha$. This might be attributed to the local decay behavior of the solution $u$: the larger is the fractional
order $\alpha$, the slower is the solution decay around $t=0$. According to Remark \ref{rem:interp}, we have
\begin{equation*}
  \| u(t_n)-U_h^n \|_{L^2(\Omega)} \le C (N^{-1} t_n^{\al\sigma} +h^{2} t_n^{-\al(1-\sigma)}) \|A^\sigma v\|_{L^2\II},
\quad \sigma\in[0,1].
\end{equation*}
Thus the temporal error deteriorates as the time $t_n\to0$ at a rate like $t_n^{\alpha}$ and $t_n^{\al/8
-\al\ep}$ for $Av\in L^2\II$ and $v\in D(A^{1/8-\epsilon})$, with $\epsilon\in(0,1/8)$, respectively.
In particular, for fixed $N$, the error behaves like $t_n^{1/2}$ and $t_n^{1/16}$ for cases
(a) and (b) with $\al=0.5$, respectively. This is clearly observed in Table \ref{tab:time-sing},
thereby showing the sharpness of the error estimates. Further, we observe that the L1
scheme fails to converge uniformly (in time) at a first order even though the initial
data $v$ in case (a) is very smooth, i.e., $v\in D(A^p)$ for any $p\ge0$, cf. Table
\ref{tab:time-sing}. This numerically confirms the observation in Remark \ref{rem:singularity}.

\begin{table}[htb!]
\caption{The $L^2$-norm of the error for the subdiffusion equation with initial data (a) and (b) at $t=0.1$ with $h=2^{-13}$, $\tau=1/N$.}
\label{tab:time-error}
\begin{center}
\vspace{-.3cm}{\setlength{\tabcolsep}{7pt}
     \begin{tabular}{|c|c|cccccc|c|}
     \hline
      $\alpha$ &  $N$  &$10$ &$20$ &$40$ & $80$ & $160$ &$320$ &rate \\
     \hline
     $0.1$ & (a)  &1.46e-4 &7.18e-5 &3.55e-5 &1.77e-5 &8.82e-6 &4.40e-6 &$\approx$ 1.01 (1.00)\\
             &  (b)  &3.95e-4 &1.93e-4 &9.57e-5 &4.76e-5 &2.38e-5 &1.19e-5 &$\approx$ 1.00 (1.00)\\
      \hline
     $0.5$ &  (a)  &1.22e-3 &5.89e-4 &2.88e-4 &1.43e-4 &7.08e-5 &3.52e-5 &$\approx$ 1.01 (1.00)\\
             &  (b)  &3.65e-3 &1.73e-3 &8.36e-4 &4.09e-4 &2.02e-4 &1.00e-4 &$\approx$ 1.02 (1.00)\\
      \hline
     $0.9$ &  (a)  &7.01e-3 &3.05e-3 &1.39e-3 &6.53e-4 &3.12e-4 &1.50e-4 &$\approx$ 1.07 (1.00)\\
             &  (b)  &1.54e-2 &7.67e-3 &3.79e-3 &1.87e-3 &9.23e-4 &4.55e-4 &$\approx$ 1.02 (1.00)\\
      \hline
     \end{tabular}}
\end{center}
\end{table}

\begin{table}[htb!]
\caption{The $L^2$-norm of the error for the subdiffusion equation with initial data (a) and (b) with
$\alpha=0.5$, $h=2^{-13}$ and $N=10$.}
\label{tab:time-sing}
\begin{center}
\vspace{-.3cm}{\setlength{\tabcolsep}{7pt}
     \begin{tabular}{|c|c|ccccc|c|c|}
     \hline
      $t$  & 1e-5 & 1e-6  & 1e-7  & 1e-8 & 1e-9 & 1e-10 & rate \\
     \hline
      (a) &2.94e-3 &1.05e-3 &3.45e-4 &1.11e-4 &3.51e-5 &1.11e-5 &$\approx$ 0.50 (0.50) \\
      \hline
      (b) &3.02e-3 &2.56e-3 &2.18e-3 &1.86e-3 &1.58e-3 &1.35e-3 &$\approx$  0.07 (0.06) \\
      \hline
     \end{tabular}}
\end{center}
\end{table}

\subsection{Time-space fractional problem.} Now we present numerical results for the space time
fractional problem. Since the exact solution is not available in closed form, we compute the reference
solution by the second-order backward difference scheme from \cite{JinLazarovZhou:2014a} on a much
finer mesh, i.e., $N=1000$. The numerical results for case (a) with different $\al$ and $\beta$ values
are presented in Table \ref{tab:time-error-a}. A first-order convergence is observed, and the
convergence rate is independent of the time- and space-fractional orders. Interestingly, the
observation is valid also for the case $\beta=5/4$, for which the theory in Section \ref{sec:space-frac} does not cover, awaiting further
study. For a fixed $\alpha$ value, the error decreases with the increase of the fractional order
$\beta$, which indicates that the solution decays faster as $\beta$ approaches two. This is also
consistent with the fact that the closer is the $\beta$ value to unity, the more singular is the solution,
and thus more challenging to approximate numerically, cf. \cite{JinLazarovPasciak:2013a}. For the
nonsmooth case (b), we are particularly interested in the case of small $t$. Thus we present the
numerical results for $t=0.1$, $t=0.01$ and $t=0.001$ in Table \ref{tab:time-error-b}. The
experiment shows the first order convergence is robust for nonsmooth data even if $t$ is close to
zero. Like before, for fixed $n$, let $t_n\to 0$, the error behaves like $t_n^{\al}$
for case (a), which is fully confirmed by the numerical results, cf. Table
\ref{tab:time-sing-spfrac}, indicating the sharpness of the estimate in Theorem \ref{thm:error_fully-sp-frac}.

\begin{table}[htb!]
\caption{The $L^2$-norm of the error for the time-space fractional problem with initial data (a) at
$t=0.1$ with $h=2^{-13}$, $\tau\to 0$.}
\label{tab:time-error-a}
\begin{center}
\vspace{-.3cm}{\setlength{\tabcolsep}{7pt}
     \begin{tabular}{|c|c|ccccc|c|}
     \hline
      $\alpha$ &  $\beta\backslash N$ &5 &$10$ &$20$ &$40$ & $80$ &rate \\
     \hline
               & $1.25$  &1.37e-3 &6.55e-4 &3.21e-4 &1.59e-4 &7.90e-5  &$\approx$ 1.01 (1.00)\\
     $0.1$  & $1.5$  &8.41e-4 &4.03e-4 &1.98e-4 &9.78e-5 &4.87e-5 &$\approx$ 1.01 (1.00)\\
               & $1.75$  &5.08e-4 &2.44e-4 &1.19e-4 &5.92e-5 &2.94e-5  &$\approx$ 1.01 (1.00)\\
      \hline
               & $1.25$  &1.52e-2 &6.69e-3 &3.12e-3 &1.50e-3 &7.32e-4 &$\approx$ 1.03 (1.00)\\
     $0.5$  & $1.5$  &8.22e-3 &3.70e-3 &1.75e-3 &8.49e-4 &4.17e-4  &$\approx$ 1.05 (1.00)\\
                & $1.75$  &4.69e-3 &2.14e-3 &1.02e-3 &4.97e-4 &2.45e-4  &$\approx$ 1.03 (1.00)\\
      \hline
               & $1.25$  &6.01e-2 &3.19e-2 &1.62e-2 &8.07e-3 &3.99e-4 &$\approx$ 1.01 (1.00)\\
     $0.9$ & $1.5$  &5.61e-2 &2.86e-2 &1.42e-2 &6.95e-3 &3.39e-3  &$\approx$ 1.03 (1.00)\\
               & $1.75$ &3.58e-2 &1.66e-2 &7.74e-3 &3.66e-3 & 1.75e-3   &$\approx$ 1.07 (1.00)\\
      \hline
     \end{tabular}}
\end{center}
\end{table}

\begin{table}[htb!]
\caption{The $L^2$-norm of the error for the time-space fractional problem with initial data (b) with
$\beta=1.5$, $h=2^{-13}$, $t=0.1$, $0.01$ and $0.001$.}
\label{tab:time-error-b}
\begin{center}
\vspace{-.3cm}{\setlength{\tabcolsep}{7pt}
     \begin{tabular}{|c|c|ccccc|c|}
     \hline
      $\alpha$ &  $t\backslash N$  &$5$ &$10$ &$20$ &$40$ & $80$ &rate \\
     \hline
               & $0.1$    &1.53e-3 &7.32e-4 &3.59e-4 &1.77e-4 &8.83e-5  &$\approx$ 1.01 (1.00)\\
     $0.1$     & $0.01$   &1.74e-3 &8.34e-4 &4.08e-4 &2.02e-4 &1.01e-4 &$\approx$ 1.01 (1.00)\\
               & $0.001$  &1.94e-3 &9.31e-4 &4.56e-4 &2.25e-4 &1.12e-4  &$\approx$ 1.01 (1.00)\\
      \hline
               & $0.1$    &1.39e-2 &6.36e-3 &3.01e-3 &1.45e-3 &7.11e-4 &$\approx$ 1.04 (1.00)\\
     $0.5$     & $0.01$   &1.22e-2 &5.86e-3 &2.85e-3 &1.40e-3 &6.89e-4  &$\approx$ 1.03 (1.00)\\
               & $0.001$  &8.02e-3 &3.83e-3 &1.86e-3 &9.12e-4 &4.50e-4  &$\approx$ 1.02 (1.00)\\
      \hline
               & $0.1$  &1.99e-2 &1.02e-2 &5.15e-3 &2.60e-3 &1.31e-3 &$\approx$ 1.00 (1.00)\\
     $0.9$ &     $0.01$ &1.21e-2 &6.10e-3 &3.05e-3 &1.52e-3 &7.53e-4  &$\approx$ 1.00 (1.00)\\
               & $0.001$&7.63e-3 &3.83e-3 &1.91e-3 &9.51e-4 &4.73e-4 &$\approx$ 1.00 (1.00)\\
      \hline
     \end{tabular}}
\end{center}
\end{table}

\begin{table}[htb!]
\caption{The $L^2$-norm of the error for the time-space fractional problem with initial data (a) and
(b), $\alpha=0.5$ and $\beta=1.5$, as $t\rightarrow 0$ with $h=2^{-13}$ and $N=5$.}
\label{tab:time-sing-spfrac}
\begin{center}
\vspace{-.3cm}{\setlength{\tabcolsep}{7pt}
     \begin{tabular}{|c|c|ccccc|c|c|}
     \hline
      $t$  & 1e-5 & 1e-6  & 1e-7  & 1e-8 & 1e-9 & 1e-10 & rate \\
     \hline
      (a) &2.60e-3 &8.90e-4 &2.96e-4 &9.71e-5 &3.17e-5 &1.03e-5 &$\approx$ 0.48 (0.50) \\
      \hline
      (b) &4.74e-3 &3.76e-3 &3.02e-3 &2.44e-3 &1.99e-3 &1.63e-3 &$\approx$  0.09 (---) \\
      \hline
     \end{tabular}}
\end{center}
\end{table}

\section{Conclusions}
In this paper we have revisited the popular L1 time-stepping scheme for discretizing the Caputo fractional derivative
of order $\alpha\in (0,1)$, arising in the modeling of subdiffusion, and rigorously established the first order convergence for both smooth and nonsmooth
initial data. This result complements existing convergence analysis, which assumes a $C^2$ regularity in time.
The extensive numerical experiments fully verify the sharpness of the estimates and robustness of the scheme. The convergence analysis is valid
for more general sectorial operators, and in particular, it covers also the space-time fractional
differential equations, for which we have also provided error estimates.

In view of the solution singularity for time $t$ close to zero, it is natural to consider the L1 scheme on
a nonuniform time mesh in order to arrive at a uniform first-order convergence. However, the generating function approach
used in our analysis does not work directly in this case, and it is an interesting open question to rigorously
establish error estimates directly in terms of the data regularity. A second interesting future research problem is the
convergence rate analysis of the scheme for the diffusion wave equation, which involves a Caputo fractional
derivative of order $\alpha\in (1,2)$ in time.

\section*{Acknowledgment}
The authors are grateful to the anonymous referee for his/her constructive comments, which have led to improved quality of the paper.
The research of B. Jin has been partially supported by NSF Grant DMS-1319052 and National Science Foundation of China No. 11471141, and that of R. Lazarov was supported in parts
by NSF Grant DMS-1016525 and also by Award No. KUS-C1-016-04, made by King Abdullah University of Science and Technology (KAUST). Z. Zhou was partially supported by NSF Grant DMS-1016525.

\bibliographystyle{abbrv}
\bibliography{frac}
\end{document}